\documentclass[]{svmult}
\usepackage{geometry}                
\usepackage{chngcntr}
\usepackage{graphicx}
\usepackage{amssymb, amsmath}
\usepackage{epstopdf}
\usepackage{enumerate}
\usepackage[lined, boxed,linesnumbered, ruled]{algorithm2e}
\newtheorem{identity}{Identity}[section]
\newtheorem{assumptions}{Assumptions}[section]
\newcommand{\brackets}[1]{\ensuremath{\left[#1\right]}}
\newcommand{\bracketsInverse}[2][1]{\ensuremath{\brackets{#2}^{-#1}}}

\newcommand{\ceilfrac}[2]{\ensuremath{\left\lceil\frac{#1}{#2}\right\rceil}}

\newcommand{\denominator}[1]{\ensuremath{\denominatorSymbol_{#1}}}
\newcommand{\denominatorSymbol}{\ensuremath{D}}

\newcommand{\equivMod}[1]{\ensuremath{\underset{#1}{\equiv}}}
\newcommand{\eulerPhi}[1]{\ensuremath{\phi\left(#1\right)}}
\newcommand{\exponent}[1]{\ensuremath{\exponentSymbol_{#1}}}

\newcommand{\exponentSymbol}{\ladicExponentSymbol}

\newcommand{\fibonacci}[1]{\ensuremath{F_{#1}}}

\newcommand{\graded}[1]{\ensuremath{{\bf #1}}}

\newcommand{\integers}{\ensuremath{\mathbb{Z}}}

\newcommand{\iterate}[1]{\ensuremath{\iterateSymbol_{#1}}}

\newcommand{\iterateSymbol}{\ensuremath{n}}

\newcommand{\ladicDigit}[1]{\ensuremath{\ladicDigitSymbol_{#1}}}
\newcommand{\ladicDigitSymbol}{\ensuremath{b}}

\newcommand{\ladicExponentSymbol}{\ensuremath{e}}
\newcommand{\lExp}[1]{\ladicExponentSymbol_{#1}}
\newcommand{\leps}[1]{\ladicExponentPrefixSum{#1}}
\newcommand{\less}[1]{\ladicExponentSuffixSum{#1}}
\newcommand{\ladicExponentPrefixSum}[1]{\ensuremath{\ladicExponentSumSymbol_{#1}}}
\newcommand{\ladicExponentSuffixSum}[1]{\ensuremath{\overline{\ladicExponentSumSymbol}_{#1}}}

\newcommand{\ladicExponentSumSymbol}{\ensuremath{E}}

\newcommand{\ladicResidueApproximation}[1]{\ensuremath{\ladicResidueApproximationSymbol_{#1}}}
\newcommand{\ladicResidueApproximationSymbol}{\ensuremath{\Lambda}}
\newcommand{\ladicResidue}[1]{\ensuremath{\ladicResidueSymbol_{#1}}}
\newcommand{\ladicResidueSymbol}{\ensuremath{\lambda}}

\newcommand{\List}[3]{\ensuremath{\left(#1_{#2},\ldots,#1_{#3}\right)}}
\newcommand{\listKernel}[3]{\ensuremath{#1_{#2},\ldots,#1_{#3}}}

\newcommand{\madicDigit}[1]{\ensuremath{\madicDigitSymbol_{#1}}}
\newcommand{\madicDigitSymbol}{\ensuremath{d}}

\newcommand{\madicExponentSymbol}{\ensuremath{f}}

\newcommand{\meps}[1]{\madicExponentPrefixSum{#1}}
\newcommand{\mess}[1]{\madicExponentSuffixSum{#1}}

\newcommand{\madicExponentPrefixSum}[1]{\ensuremath{\madicExponentSumSymbol_{#1}}}
\newcommand{\madicExponentSuffixSum}[1]{\ensuremath{\overline{\madicExponentSumSymbol}_{#1}}}
\newcommand{\madicExponentSumSymbol}{\ensuremath{F}}

\newcommand{\madicResidueApproximation}[1]{\ensuremath{\madicResidueApproximationSymbol_{#1}}}
\newcommand{\madicResidueApproximationSymbol}{\ensuremath{M}}
\newcommand{\madicResidue}[1]{\ensuremath{\madicResidueSymbol_{#1}}}
\newcommand{\madicResidueSymbol}{\ensuremath{\mu}}

\newcommand{\naturals}{\ensuremath{\mathbb{N}}}
\newcommand{\negPower}[1]{\left(-\right)^{#1}}

\newcommand{\nOpenSet}[1]{\ensuremath{\left[#1\right)}}

\newcommand{\numerator}[1]{\ensuremath{\numeratorSymbol_{#1}}}

\newcommand{\numeratorSymbol}{\ensuremath{N}}
\newcommand{\nZeroOpenSet}[1]{\ensuremath{\left[#1\right)_0}}
\newcommand{\nSet}[1]{\ensuremath{\left[#1\right]}}
\newcommand{\nZeroSet}[1]{\ensuremath{\left[#1\right]_0}}

\newcommand{\parentheses}[1]{\ensuremath{\left(#1\right)}}

\newcommand{\pfrac}[2]{\ensuremath{\left(\frac{#1}{#2}\right)}}

\newcommand{\resRep}[2]{\ensuremath{\brackets{#1}_{#2}}}
\newcommand{\resRepInv}[2]{\ensuremath{\resRep{#1}{#2}^{-1}}}

\newcommand{\translationMatrixEntrySymbol}{\ensuremath{a}}
\newcommand{\translationValue}[1]{\ensuremath{\translationMatrixEntrySymbol_{#1}}}
\newcommand{\translationValueSymbol}{\ensuremath{\translationMatrixEntrySymbol}}
\newcommand{\twoSummandTerm}[2]{\ensuremath{\sigma\left(#1,#2\right)}}
\newcommand{\twoSumTerm}[1]{\ensuremath{\mathcal{S}\left(#1\right)}}
\newcommand{\twoTailSum}[1]{\ensuremath{\widehat{\Lambda}_{#1}}}

\title{A Dual-Radix Approach to Steiner's 1-Cycle Theorem}
\author{Andrey Rukhin}

\institute{{Andrey Rukhin\\
email: {\tt andrey.rukhin@gmail.com}\\This work was supported by the Naval Surface Warfare Center Dahlgren
Division's In-House Laboratory Independent Research Program.}}

\newcounter{assumptionCounter}
\stepcounter{assumptionCounter}
\begin{document}
\maketitle


\begin{abstract}{}This article presents  three  algebraic proofs of Steiner's {\sl 1-Cycle Theorem} \cite{Steiner} within the context of the (accelerated) $3x+1$ dynamical system.  Furthermore, under an assumption of an exponential upper-bound on the iterates,  the article demonstrates that the only $1$-cycles in the (accelerated) $3x-1$ dynamical system are $(1)$ and $(5,7)$. 
\end{abstract}

\section{Introduction}
Within the context of the $3x+1$ Problem, {Steiner's \sl 1-cycle Theorem} \cite{Steiner} is a result pertaining to the non-existence of {\sl $1$-cycles} (or {\sl circuits}): for all $a,b\in \mathbb{N}$, Steiner shows that a rational expression of the form 
\begin{equation}\label{steinerRatio}
 \frac{2^a - 1}{2^{a+b}-3^b} \end{equation}
  does not assume a positive integer value except in the case where $a = b = 1$. In the proof, the author appeals to the continued fraction expansion of $\log_2 3$, transcendental number theory, and extensive numerical computation (see \cite{SimonsDeWeger}).  This argument serves as the basis for demonstrating the non-existence of $2$-cycles in \cite{Simons2Cycles},  and  the non-existence of $m$-cycles  in  \cite{SimonsDeWeger} where $m\leq 68$.  
  
  	The result has been strengthened in \cite{Brox} as follows:   Let $C$ denote a cycle in the (accelerated) $3x+1$ dynamical system $T: 2\integers+1 \to 2\integers+1$, defined by the mapping 
  	\[
  	T(x) = \frac{3x + 1}{2^{e(x)}}
  	\] where $e(x)$ is the $2$-adic valuation of the quantity $3x+1$.   If $e(x) \geq 2$,  the element $x$ is said to be a {\sl descending element} in $C$, and we define $\delta(C)$ to be the number of descending elements in $C$.     Theorem 1.1 in \cite{Brox} demonstrates that the number of cycles satisfying the inequality
  $
    \delta(C) < 2\log\parentheses{|C|}
  $ is finite; Steiner's result addresses the case where $d(C) = 1$ by showing that the only (accelerated) cycle with a single descending element is the cycle including $1$.

 However, the author in \cite{LagariasATT} declares that the ``most remarkable thing about [Steiner's theorem] is the weakness of its conclusion compared to the strength of the methods used in its proof." This article offers alternative proofs of this theorem by demonstrating the non-integrality of  the maximal element of a $1$-cycle
 \[
\frac{(2^{a+1}+1)3^{b-1} - 2^{a+b}}{2^{a+b}- 3^{b}} = 2\cdot3^{b-1}\pfrac{2^{a}-1}{2^{a+b}-3^{b}} - 1
\] within a variety of algebraic settings.
 Assuming the upper bound on periodic iterates established in \cite{BelagaMignotte}, these proofs exploit that fact that the denominator in the above expression is coprime to both $2$ and $3$.
 Based on the results in \cite{Rukhin18}, the first proof appeals to elementary modular arithmetic,  the second proof exploits identities on weighted binomial coefficients and the Fibonacci numbers, and  the third proof analyzes the $2$-adic and $3$-adic digits of the values in a $1$-cycle.

The article concludes with a similiar analyses of the existence of $1$-cycles within the (accelerated) $3x-1$ dynamical system: we will demonstrate that, under the assumption of an exponential upper bound on the iterate values of a periodic orbit, the only $1$-cycles are $(1)$ and $(5,7)$.

\section{Overview}

\subsection{Notation}

This manuscript inherits all of the notation and definitions established in \cite{Rukhin18}, which we summarize here.  Let $\tau \in \naturals$, and  let $\graded{e},\graded{f} \in \naturals^{\tau}$ where $\graded{e} = \List{\ladicExponentSymbol}{0}{\tau-1}$ and $\graded{f} = \List{\madicExponentSymbol}{0}{\tau-1}$.  For each $u\in \integers$, define $\leps{u} = \sum_{0\leq w < u} \lExp{w \bmod \tau}$ and  $\less{u} = \sum_{0\leq w < u} \lExp{(\tau-1-w) \bmod \tau}$; we will define $\meps{u}$ and $\mess{u}$ in an analogous manner with the elements of $\graded{f}$.

For a positive integer $b$, we will write $\nSet{b} = \{1,\ldots,b\}$ and $\nOpenSet{b} = \{1,\ldots, b-1\}$; furthermore, we will write $\nZeroSet{b} = \nSet{b} \cup \{0\}$ and $\nZeroOpenSet{b} = \nOpenSet{b} \cup \{0\}$. 

 For any integer $a$ and positive base $b \ (b\geq 1)$, let 
$
[a]_b
$ denote the element\footnote{This element is also known as the {\sl standard} (or {\sl canonical}) {\sl representative} of the equivalence class $\overline{a} \bmod b$.} of $\nZeroOpenSet{b}$ that satisfies the equivalence
$[a]_b \equiv a \bmod b$.  
We will also write $\resRepInv{a}{b}$ to denote the element in $\nZeroOpenSet{b}$ that satisfies the equivalence $\resRep{a}{b}\resRepInv{a}{b} \equivMod{b}1$.  

For the maximal iterate value $\iterate{\max}$ within a $1$-cycle, we will define 
$
\madicResidue{\tau} = \iterate{\max} \bmod 3^{\tau}
$ and
$
\ladicResidue{\tau} = \iterate{\max} \bmod 2^{e+\tau-1}
$ for $\ladicExponentSymbol,\tau \in \naturals$ 

We will write $\negPower{u}$ to denote the quantity $(-1)^u$ for each $u\in \naturals_0$.
%

\subsection{Argument Overview}

The dual-radix approach to the non-existence of circuits is based upon the following premises:

\begin{enumerate}[i.]
%

\item   We will establish an upper bound of $3^{\tau}$ for a potential, periodic iterate value over $\naturals$ for the (accelerated) $3x+1$ Problem.  In this context,   the authors in \cite{BelagaMignotte}  have demonstrated that the maximal iterate $\iterate{\max}$ within a periodic orbit admits the upper bound
\begin{equation}\label{iterateUpperBound}
\iterate{\max}  < \frac{\pfrac{3}{2}^{\tau-1}}{1 - \frac{3^{\tau}}{2^{\less{\tau}}}}\leq  \tau^C\pfrac{3}{2}^{\tau-1} = o\parentheses{3^{\tau-1}}
\end{equation} for some effectively computable constant $C$ (by applying the result in \cite{BakerWustholz}). A recent upper bound on $C$  is available in \cite{Rhin}, in which the author establishes the inequality\footnote{In their notation, we set $u_0 = 0$, $u_1 = -\less{\tau}$, and $u_2 = \tau$.}
\begin{equation}\label{rhinInequality}
\left|-\less{\tau}\log 2 +\tau \log 3\right| \geq \less{\tau}^{-13.3};
\end{equation}  consequently, assuming $2^{\less{\tau}}>3^{\tau}$, we can bound\footnote{We can shed the logarithms:  when $|w| < 1$,  the  power series expansion of $\log (1+w) = \sum_{u\geq 1}(-1)^{u-1}\frac{w^u}{u}$ yields $|\log(1+w)| \leq 2|w|$ when $|w| \leq \frac{1}{2}$. See \cite{Evertse} (Corollary 1.6). }  the denominator in (\ref{iterateUpperBound}) from below
$
1 - \frac{3^{\tau}}{2^{\less{\tau}}} \geq  \frac{\less{\tau}^{-13.3}}{2}.
$ According to \cite{Eliahou}, for a periodic orbit over $\naturals$  of length $\less{\tau}$, the ratio $\frac{\less{\tau}}{\tau}$ satisfies the inequality
\[
\frac{\less{\tau}}{\tau} \leq \lg\parentheses{3+\frac{1}{\iterate{\min}}} \leq 2; 
\]   numerical computation yields
$
\iterate{\max}< \pfrac{3}{2}^{\tau-1} 2\cdot (2\tau)^{13.3} < 3^{\tau} 
$ when $\tau \geq 103$.  
 
 Thus, if $\iterate{\max}> 3^{\tau}$ and $\iterate{\max}\in \naturals$, then $\tau < 103$. However, the author in \cite{Garner} demonstrates that the length of a non-trivial periodic orbit (excluding $1$) over \naturals\ must satisfy the inequality $2\tau \geq \less{\tau} \geq 35,400$.

 Thus, if $\iterate{\max}\in \naturals$, then $\iterate{\max} < 3^{\tau} < 2^{\less{\tau}}$, and the equalities
$
\iterate{\max} = \madicResidue{\tau} =  \ladicResidue{\tau}
$ must hold.
\item Within a circuit of order $\tau$ in the (accelerated) $3x+1$ dynamical system,  the maximal element equals
\[
\frac{(2^e+1)3^{\tau-1} - 2^{e+\tau-1}}{2^{e+\tau-1}- 3^{\tau}} = 2\cdot3^{\tau-1}\pfrac{2^{e-1}-1}{2^{e+\tau-1}-3^{\tau}} - 1
\] for some $e\in \naturals$ (see \cite{BohmSontacchi78}).

When $\tau=1$, we note that $2^{e}-3   
\geq 2^{e-1}-1$ for $e\geq 2$; thus the ratio in (\ref{steinerRatio}), evaluated at $a=e-1$ and $b=1$, is at most one. When $e=1$, the left-hand side of the equality above is negative, and the ratio in (\ref{steinerRatio}) vanishes.

  When $\tau>1$, we will analyze the difference of canonical residues
\[
\madicResidue{\tau} = \brackets{(2^e+1)3^{\tau-1} - 2^{e+\tau-1}}[2^{e+\tau-1}]^{-1} \bmod 3^{\tau}
\] and
\[
\ladicResidue{\tau} = \brackets{(2^e+1)3^{\tau-1} - 2^{e+\tau-1}}[-3^{\tau}]^{-1} \bmod 2^{\less{\tau}};
\] we will demonstrate the inequality $\madicResidue{\tau} \neq \ladicResidue{\tau}$  (contradicting the assumption that  $\iterate{\max} = \madicResidue{\tau} = \ladicResidue{\tau}$  as per above).

  We will also perform similar analyses on the maximal element of a circuit within the (accelerated) $3x-1$ dynamical system; we will show that, assuming\footnote{Appealing to a similar argument outlined abve, this condition holds for finitely many $\tau$ for each fixed $e\in \naturals$.} the inequality  $\iterate{\max} < 2^{\less{\tau}}$,  a circuit over $\naturals$ exists if and only if either $e=1$, or $\tau=e=2$. 
\end{enumerate}

\section{Circuits with the $3x+1$ Dynamical System}
 Throughout the remainder of the manuscript, unless otherwise stated, we assume that
\begin{enumerate}[i.]
	\item  $\tau \in \naturals$ with $\tau\geq 2$;
	\item  $\graded{f} = \List{1}{}{} \in \naturals^{\tau} $;
	\item  $\graded{e}  = (\underbrace{\listKernel{1}{}{}}_{\tau-1},e)$  for some $e\in \naturals$; and
	\item $\graded{a} = \List{\translationValueSymbol}{0}{\tau-1} \in \{-1,+1\}^{\tau}$.
\end{enumerate}  

We begin with the following assumptions.
  
\begin{assumptions}[\arabic{chapter}.\arabic{section}.\arabic{assumptionCounter}] \label{hypotheses::threeEcksPlusOneResiduesOfMaxIterate}Assume  3.1 and  3.3 from \cite{Rukhin18},  and let 
$\graded{a} = \graded{1}^{\tau}$. 
  Let $N  = (2^e+1)3^{\tau-1} - 2^{e+\tau-1}$, and
  let $D = 2^{e+\tau-1} - 3^{\tau}$ where $D>0$.
  
Assume that $$\iterate{\max} = \frac{N}{D}  < \min\left(3^{\tau}, 2^{\less{\tau}}\right),$$  let
$\madicResidue{\tau} = \iterate{\max}\bmod 3^{\tau}$, and let $\ladicResidue{\tau} = \iterate{\max}\bmod 2^{e+\tau-1}$.
\end{assumptions} 
\stepcounter{assumptionCounter}

Under these assumptions, if $\iterate{\max} \in \naturals$, then the chain of equalities 
$\iterate{\max} =  \madicResidue{\tau} = \ladicResidue{\tau}$ holds.

Our goal for the remainder of this subsection is to prove the following theorem:

\begin{theorem}\label{theorem::threeEcksPlusOneResidues}
Assume  (\ref{hypotheses::threeEcksPlusOneResiduesOfMaxIterate}).

We have the equalities 
\begin{equation}
\madicResidue{\tau} = 
	\begin{cases}
		3^{\tau-1}-1 & e\equivMod{2} 0\\
		3^{\tau}-1& e\equivMod{2} 1
	\end{cases}\notag
\end{equation} when $\tau \equivMod{2} 0$, and
\begin{equation}
\madicResidue{\tau} = 
	\begin{cases}
		2\cdot 3^{\tau-1}-1 & e\equivMod{2} 0 \\
		3^{\tau}-1& e\equivMod{2} 1 
	\end{cases}\notag
\end{equation} when $\tau \equivMod{2} 1$.  

Furthermore, when  $\tau \equivMod{2} 1 \equivMod{2} e-1$, then
\[
\ladicResidue{\tau} = 2^e\pfrac{2^{\tau-1}-1}{3} + \frac{2^{e+\tau-1}-1}{3}= \frac{(2^{\tau}-1)2^e-1}{3}.
\]
\end{theorem}

For completeness, we have
 \begin{equation}
		\ladicResidue{\tau} = 
		\begin{cases}
			\frac{(2^{\tau-1}-1)2^e-1}{3}  & e\equivMod{2} 0\\
			2^{e+\tau-1}-\frac{2^{e}+1}{3}& e\equivMod{2} 1
		\end{cases}\notag
	\end{equation} when $\tau \equivMod{2} 0$, and
	\begin{equation}
		\ladicResidue{\tau} = 
		\begin{cases}
			\frac{(2^{\tau}-1)2^e-1}{3} & e\equivMod{2} 0 \\
			2^{e+\tau-1}-\frac{2^{e}+1}{3}& e\equivMod{2} 1 
		\end{cases}\notag
	\end{equation} when $\tau \equivMod{2} 1$.
However, in order to expedite the proofs,   we exclude three out of the four cases when the corresponding canonical $3$-residue $\madicResidue{\tau}$ is even (assuring the inequality $\madicResidue{\tau}\neq \ladicResidue{\tau}$).  We exclude the remaining case with the following lemma.

\begin{lemma}\label{Lemma::ThreeEcksPlusOneResidueCondition}
Assume that $\tau \equivMod{2} 1 \equivMod{2} e-1$; furthermore, let
$
\madicResidue{\tau} = 2\cdot3^{\tau-1} - 1,
$ and
$
\ladicResidue{\tau} = \frac{(2^{\tau}-1)2^e-1}{3}.
$
Then, the inequality $\madicResidue{\tau} \neq \ladicResidue{\tau}$ holds.
\end{lemma}

\begin{proof}

By way of contradiction, assume that the natural number $e$ satisfies the equality
$
2\cdot3^{\tau-1} - 1 = \frac{(2^{\tau}-1)2^e-1}{3};
$ equivalently, we require that the equality
$
2\left(3^{\tau} - 1\right) = (2^{\tau}-1)2^e
$  holds. However, we have that
$$
2^{e-2}\parentheses{2^{\tau}-1} = \frac{3^{\tau}-1}{2} \equivMod{2} \sum_{0\leq w < \tau}3^w  \equivMod{2}  1
$$  for all odd, positive $\tau$.  When $e=2$, the value of $\tau$ must satisfy the equality
$
2 - \frac{1}{2^{\tau}} = \pfrac{3}{2}^{\tau};
$  however, this equality fails to hold for $\tau>1$.

\end{proof}

Lemma \ref{Lemma::ThreeEcksPlusOneResidueCondition},  Assumptions (\ref{hypotheses::threeEcksPlusOneResiduesOfMaxIterate}),  and Theorem \ref{theorem::threeEcksPlusOneResidues}, along with the bounds provided in \cite{SimonsDeWeger}, \cite{Eliahou}, and \cite{Garner},  demonstrate the non-existence of circuits in the $3x+1$ dynamical system.
%
%
%
%
%


\subsection{Elementary Modular Arithmetic}
	Our first proof of Theorem \ref{theorem::threeEcksPlusOneResidues} appeals to elementary modular arithmetic.

\begin{proof}

We will write
\[
\madicResidue{\tau} \equivMod{3^{\tau}} ND^{-1} 
 \equivMod{3^{\tau}} \brackets{(2^e+1)3^{\tau-1} - 2^{e+\tau-1}}\resRepInv{2^{e+\tau-1}}{} 
  \equivMod{3^{\tau}}\brackets{\resRepInv{2^{\tau-1}}{3^1} + \resRepInv{2^{e+\tau-1}}{3^1}} 3^{\tau-1} - 1.
\]
It follows that
$
\madicResidue{\tau} \equivMod{3^{\tau}} 3^{\tau-1}\negPower{\tau-1}\brackets{1 + \negPower{e}} -1.
$
Thus, when $e\equivMod{2}  1$, we have $\madicResidue{\tau} = 3^{\tau}-1 \equivMod{2} 0$.  Similarly, when $e\equivMod{2}  0$ and $\tau\equivMod{2} 0$, we have $\madicResidue{\tau} = 3^{\tau-1}-1 \equivMod{2} 0$. 
When $\tau \equivMod{2}1 \equivMod{2} e-1$, we arrive at the equality
$
\madicResidue{\tau} = 2\cdot3^{\tau-1} - 1.
$

For the $\graded{2}$-remainder, we begin by writing
\[
\ladicResidue{\tau} \equivMod{2^{e+\tau-1}} ND^{-1} 
 \equivMod{2^{e+\tau-1}}  \brackets{(2^e+1)3^{\tau-1} - 2^{e+\tau-1}}\resRepInv{-3^{\tau}}{} 
 \equivMod{2^{e+\tau-1}} 2^e\resRepInv{-3}{2^{\tau-1}} + \resRepInv{-3}{}.
\]
When $\tau \equivMod{2}1 \equivMod{2} e-1$, we have $\resRep{-3^1}{2^{\tau-1}}^{-1} = \frac{2^{\tau-1}-1}{3}$ and  $\resRep{-3^1}{2^{e+\tau-1}}^{-1} = \frac{2^{e+\tau-1}-1}{3}.$

As
\[
2^e\pfrac{2^{\tau-1}-1}{3} + \frac{2^{e+\tau-1}-1}{3}  = \frac{2\parentheses{2^{e+\tau-1}}-2^e-1}{3} < 2^{e+\tau-1},
\] we arrive at the chain of equalities
$
\ladicResidue{\tau} =2^e\pfrac{2^{\tau-1}-1}{3} + \frac{2^{e+\tau-1}-1}{3}= \frac{(2^{\tau}-1)2^e-1}{3}.
$
\end{proof}

\subsection{Weighted Binomial Coefficients}{}

The previous approach is apparently limited; it is unclear to the author how to extrapolate this approach to admissible sequences of order $\tau$ with an arbitrary $\graded{2}$-grading $\List{\ladicExponentSymbol}{0}{\tau-1}$.  In this subsection, we introduce a more robust approach to identifying the $3$-residues and $\graded{2}$-remainders of the iterates of an admissible cycle in a $(3,2)$-system.   Moreover, we do so by connecting the residues of $(3,2)$-systems to the well-known {\sl Fibonacci sequence} by way of elementary equivalence identities, which we establish first.

\begin{lemma}\label{lemma::multinomialModularIdentity}
For $a,b,z\in \naturals$, the equivalence
\[
\parentheses{\sum_{0\leq w < b}z^w}^a \equivMod{z^b} \sum_{0\leq w < b}{a-1+w\choose w} z^w
\]
holds.
\end{lemma}

\begin{proof}

	Define $S_b(z) = \sum_{0\leq w<b} z^w$, and define $T_{a,b}(z) = \sum_{0\leq w < b}{a-1+w\choose w} z^w$.
	The proof is by induction on $b$.  
	
	When $b=1$, we arrive at the equivalence
	$
	1^a \equivMod{z} {a-1\choose 0}
	$ for all $a,z\in \naturals$.
	
	Assume the claim holds for $b\in \naturals$.  The identity $S_{b+1}(z) = zS_{b}(z)+1$ allows the chain of equivalences
\[
	\brackets{S_{b+1}(z)}^a  \equivMod{z^{b+1}} \sum_{0\leq y < b+1} {a\choose y} z^{y}\brackets{S_{b}(z)}^y 
	\equivMod{z^{b+1}} {a\choose 0}z^0 + \sum_{1\leq y < b+1} {a\choose y} z^{y} T_{y,b}(z) .
\] We will recast the coefficient of $z^0$ as ${a-1\choose 0}$, and we will write
	\[
	 \sum_{1\leq y < b+1} {a\choose y} z^{y} T_{y,b}(z) = \sum_{1\leq y < b+1} \sum_{0\leq u < b}z^{u+y} {a\choose y} {y-1+u\choose u} .
	\]
	For each $w\in \nOpenSet{b+1}$, the coefficient of $z^w$ is $\sum_{1\leq y \leq w} {{a\choose y}{w-1\choose w-y}} = \sum_{0\leq y < w} {{a\choose w-y}{w-1\choose y}}$, which equals
	$
	{a-1+w \choose w}
	$ as per the{\sl Vandermonde-Chu} identity.
\end{proof}

\begin{identity}[Fibonacci Identity]\label{identity::fibonacci}
Let $\fibonacci{0}=0$, $\fibonacci{1}=1$, and $\fibonacci{n}=\fibonacci{n-1}+\fibonacci{n-2}$ for $n\geq 2$.  The equality
$
\fibonacci{n} = \sum_{0\leq k < n}{n-1-k\choose k}
$ holds.
\end{identity}

We will use these identities to establish the {\sl remainder approximation functions}.

\begin{lemma}
Define the map $\madicResidueApproximation{\tau}: \naturals^{\tau}\times\naturals^{\tau} \to \integers$ to be 
\[
\madicResidueApproximation{\tau} = \madicResidueApproximation{\tau}\left(\graded{e},\graded{a}\right) = \sum_{0\leq w < u}(-)^{\leps{w+1}}3^{w}\translationValue{w}\sum_{0\leq y < \tau-w}{\leps{w+1}-1+y\choose y}3^y,
\] and define  the map $\ladicResidueApproximation{\tau}: \naturals^{\tau}\times\naturals^{\tau} \to \integers$ to be
\[
\ladicResidueApproximation{\tau} = \ladicResidueApproximation{\tau}\left(\graded{e},\graded{a}\right) =  \sum_{0\leq w < \tau}(-)^{w}2^{\less{w}}\translationValue{\tau-1-w}\sum_{0\leq y < \eta_{w} }{w+y\choose y}4^y,
\] where
$
\eta_{w}= \ceilfrac{\leps{\tau-w}}{2}.
$  

Then, the equivalences 
$
\madicResidueApproximation{\tau} \equivMod{3^{\tau}}  \madicResidue{\tau}
$ and
$
\ladicResidueApproximation{\tau}   \equivMod{2^{\less{\tau}}} \ladicResidue{\tau}
$  hold.

\end{lemma}

\begin{proof}
We will make use of the following elementary identities involving {\sl Euler's totient function}  $\phi$\/: we have 
$
3^{\eulerPhi{2}}-1 = 2
$ and 
$
2^{\eulerPhi{3}}-1 = 3.
$
In light of these identities, we will appeal to Lemma \ref{lemma::multinomialModularIdentity}:  for $a,b\in \naturals$, we will  write
\[
\bracketsInverse{2^a} \equivMod{3^b} \pfrac{1 - 3^{\eulerPhi{2}\ceilfrac{b}{\eulerPhi{2}}}}{2}^a  \equivMod{3^b} \parentheses{-}^a\parentheses{\sum_{0\leq y < b}3^y}^a \equivMod{3^b}  \parentheses{-}^a\sum_{0\leq y < b}{a-1+y\choose y}3^y,
\] and 
\[
\bracketsInverse{3^b} \equivMod{2^a} \pfrac{1 - 2^{\eulerPhi{3}\ceilfrac{a}{\eulerPhi{3}}}}{3}^b  \equivMod{2^a} \parentheses{-}^b\parentheses{\sum_{0\leq y < \ceilfrac{a}{2}}4^y}^b \equivMod{2^a}  \parentheses{-}^b\sum_{0\leq y < \ceilfrac{a}{2}}{b-1+y\choose y}4^y.
\] 

	We derive the 3-remainder approximation function as follows:
\[
	 	\madicResidue{\tau} \equivMod{3^{\tau}} \resRep{\numerator{}\denominator{}^{-1}}{3^{\tau}}  \equivMod{3^{\tau}} \sum_{0\leq w < \tau} 3^w2^{\less{\tau-1-w}}\translationValue{w}\bracketsInverse{2^{\less{\tau}}} 
		\equivMod{3^{\tau}} \sum_{0\leq w < \tau}(-)^{\leps{w+1}}3^{w}\translationValue{w}\sum_{0\leq y < \tau-w }{\leps{w+1}-1+y\choose y}3^y.
\]

  We derive the \graded{2}-remainder approximation function analogously:
\[
	 	\ladicResidue{\tau}  \equivMod{2^{\less{\tau}}} \sum_{0\leq w < \tau} 3^w2^{\less{\tau-1-w}}\translationValue{w}\bracketsInverse{-3^{\tau}} 
		\equivMod{2^{\less{\tau}}} \sum_{0\leq w < \tau}(-)^{w}2^{\less{w}}\translationValue{\tau-1-w}\sum_{0\leq y < \eta_w }{w+y\choose y}4^y.
\]
$\hspace{20pt}$
\end{proof}
It will prove useful to re-index these double-sums: for example, in the $3$-residue approximation, for each fixed $w\in \nZeroOpenSet{\tau}$ the coefficient of $3^w$ is 
\[
S_w = \sum_{0\leq y \leq w}(-)^{\leps{y+1}}{\leps{y+1}-1+w-y\choose w-y}\translationValue{y};
\] thus, we can write
$
\madicResidueApproximation{\tau} = \sum_{0\leq w < \tau}3^wS_w.
$ 

The following example illustrates the  connection between an orbit over $\naturals$ within the $3x+1$ dynamical system and the Fibonacci Sequence.

\subsubsection{Example: The $(1,4,2)$-Orbit in the $3x+1$ Dynamical System}

For this example, define $\exponent{y} = 2$ and $\translationValue{y} = 1$ for each $y\in \nZeroOpenSet{\tau}$; thus, the  sum $\leps{y+1} = 2(y+1)  \equivMod{2} 0$.  We can express the $3$-remainder approximation as
$
\madicResidueApproximation{\tau} = \sum_{0\leq w < \tau} 3^wS_w,
$ where
\[
S_w := \sum_{0\leq y \leq w}\negPower{2(y+1)}{2(y+1)-1 +w-y\choose w-y} = \sum_{0\leq y \leq w}{2w+1-y\choose y}.
\]
The sequence $\parentheses{S_w}_{w\geq 0}$ is the even-indexed bisection of the Fibonacci sequence $\parentheses{F_w}_{w\geq 0}$ as per Identity \ref{identity::fibonacci};  we have
$
S_w = F_{2(w+1)}
$ for $w\geq 0$.
It is known\footnote{OEIS:A001906} that this bisection
satisfies the recurrence\footnote{We assume the standard definition  $\fibonacci{-u} = (-)^{u-1}\fibonacci{u}$ for $u\in\naturals$.}
$
F_{2w} = 3F_{2(w-1)} - F_{2(w-2)}
$ for $w\geq 0$;  thus, we will write $\madicResidueApproximation{\tau} = \sum_{0\leq w < \tau}3^wS_w
= \sum_{0\leq w < \tau}3^w\fibonacci{2(w+1)}$, and we continue by writing
\[
\sum_{0\leq w < \tau}3^w\brackets{3\fibonacci{2w} - \fibonacci{2(w-1)}}
=  \sum_{0\leq w < \tau-1}3^{w+1}\fibonacci{2w} + 3^{\tau}\fibonacci{2(\tau-1)}- \fibonacci{-2} - \sum_{1\leq w < \tau}3^w\fibonacci{2(w-1)} 
= 3^{\tau}F_{2(\tau-1)} + 1.
\]


For  the \graded{2}-remainder approximation, we have the equalities
$
\ladicResidueApproximation{\tau} =  \sum_{0\leq w < \tau} 4^w\sum_{0\leq y \leq w}{w\choose y}(-1)^y = \sum_{0\leq w < \tau} 4^w(1-1)^w = 1
$ for $\tau\in \naturals$.

The Fibonacci sequence appears within the $\graded{2}$-remainder approximation for the following proof of Theorem \ref{theorem::threeEcksPlusOneResidues}.   In order to expedite the derivation of this \graded{2}-remainder, we will first prove the following lemma.

\begin{lemma}\label{lemma::FibonacciIdentity}
For $a\in \naturals_0$,   let $\fibonacci{a}$ denote the $a$-th Fibonacci number; furthermore, for $k\in \naturals_0$, define $\twoSummandTerm{a}{k} = 2{a+1\choose k} - {a\choose k}$, and define $\twoSumTerm{k} = \sum_{0\leq i < k}\twoSummandTerm{2k-i}{i+1}$. 

For $k\in \naturals_0$, the equality
$
	\twoSumTerm{k} = \fibonacci{2k+2} + 2\fibonacci{2k+1}-3
$ holds.
\end{lemma}

\begin{proof}  Assume the conditions within the statement of the lemma.  For $k=0$, we have $\twoSumTerm{k} = 0 = \fibonacci{2} + 2\fibonacci{1}-3$. When $k>0$, we will write
\begin{alignat}{2}
\twoSumTerm{k} &= \sum_{0\leq i < k}\brackets{2{2k-i+1\choose i+1} - {2k-i\choose i+1}} \notag\\
&=\sum_{1\leq i < k+1}\brackets{2{2k+2-i\choose i} - {2k+1-i\choose i}} \notag\\
&=2\brackets{\fibonacci{2k+3}-{2k+2\choose 0} - {k+1\choose k+1}} - \brackets{\fibonacci{2k+2}-{2k+1\choose 0}}\notag\\
&= \fibonacci{2k+2} + 2\fibonacci{2k+1}-3.\notag
\end{alignat}

\end{proof}
We proceed with the proof of the theorem.

\begin{proof}

First, we will  demonstrate the equality
$
\madicResidueApproximation{\tau} = -1 + 3^{\tau-1}\negPower{\tau-1}\brackets{1 + \negPower{e}};
$ afterwards, when assuming $\tau \equivMod{2} 1 \equivMod{2} e-1$, we will show that
\[
\ladicResidueApproximation{\tau} = 2^e\pfrac{2^{\tau-1}-1}{3} + \frac{2^{e+\tau-1}-1}{3} + 2^{e+\tau-1}\parentheses{\fibonacci{\tau-2}-1}.
\]

 In circuits, we have
\begin{equation}
	\leps{w} = 
	\begin{cases}
			w & w < \tau \\
			e+\tau-1 & w = \tau,
	\end{cases} \notag
\end{equation}
for $w \in \nOpenSet{\tau}$. Thus, when $w < \tau-1$, we have 
\begin{alignat}{2}
S_w &= \sum_{0\leq y \leq w}(-)^{\leps{y+1}}{\leps{y+1}- 1 + w-y\choose w-y} \notag \\
&= \sum_{0\leq y \leq w}(-)^{y+1}{w\choose w-y} \notag \\
&= -\sum_{0\leq y \leq w} (-)^{w-y}{w\choose y} \notag\\
&= -(1-1)^w\notag\\
&= \begin{cases}
	0 & w > 0\\
	-1 & w = 0.
\end{cases}; \notag
\end{alignat} when $w=\tau-1 \geq 1$, we have
\begin{alignat}{2}
S_{\tau-1} &= \sum_{0\leq y \leq \tau-1}(-)^{\leps{y+1}}{\leps{y+1}- 1 + \tau-1-y\choose \tau-1-y} \notag \\
&= \sum_{0\leq y \leq \tau-2}(-)^{y+1}{\tau-1\choose \tau-1-y} +  (-)^{e+\tau-1}{e+\tau-2\choose 0}\notag \\
&= -(1-1)^{\tau-1} + (-)^{\tau-1}{\tau-1\choose \tau-1} +  (-)^{e+\tau-1}{e+\tau-2\choose 0} \notag\\
&=\negPower{\tau-1}\brackets{1 + \negPower{e}}. \notag
\end{alignat}
It follows that
$
\madicResidueApproximation{\tau} = -1 + 3^{\tau-1}\negPower{\tau-1}\brackets{1 + \negPower{e}}.
$
Thus, when $e\equivMod{2}  1$, we have $\madicResidue{\tau} = 3^{\tau}-1$.  Similarly, when $e\equivMod{2}  0$ and $\tau\equivMod{2} 0$, we have $\madicResidue{\tau} = 3^{\tau-1}-1$.

When $\tau \equivMod{2}1 \equivMod{2} e-1$, we arrive at the equality
$
\madicResidue{\tau} = 2\cdot3^{\tau-1} - 1.
$  
Continuing with these parity conditions, we let $T_w$ denote the sum $\sum_{0\leq y < \ceilfrac{\leps{\tau-w}}{2}} {w+y\choose y}4^y$.
We write
\begin{alignat}{2}
\ladicResidueApproximation{\tau} &= \sum_{0\leq w < \tau} (-)^w2^{\less{w}} T_w\notag\\
&= T_0 + \sum_{1\leq w < \tau} (-)^w2^{\less{w}} T_w\notag\\
&= \sum_{0\leq y < \frac{e+\tau-1}{2}}{y\choose y}4^y + \sum_{1\leq w < \tau} (-)^w2^{\less{w}}{w\choose 0} +  \sum_{1\leq w < \tau} (-)^w2^{\less{w}}\brackets{T_w - {w\choose 0}}.\notag
\end{alignat}
We proceed with the first two sums in the final expression.   When $e+\tau-1\equivMod{2} 0$, we will write
\[
T_0 = \sum_{0\leq y < \frac{e+\tau-1}{2}}{y\choose y}4^y = \frac{2^{e+\tau-1}-1}{3}.
\] In circuits, we have
$\less{w} = e + w-1$ for $w \in \nOpenSet{\tau}$; thus, when $\tau-1\equivMod{2} 0$, we will also write
\begin{alignat}{2}
\sum_{1\leq w < \tau} (-)^w2^{\less{w}}{w\choose 0}  &\equivMod{2^{e+\tau-1}} 2^e\sum_{0\leq w < \tau-1} (-)^{w+1}2^{w} \notag\\
&\equivMod{2^{e+\tau-1}} 2^e\sum_{0\leq w < \frac{\tau-1}{2}} \brackets{2^{2w+1} - 2^{2w}} \notag\\
&\equivMod{2^{e+\tau-1}} 2^e\sum_{0\leq w < \frac{\tau-1}{2}} 4^w \notag\\
&\equivMod{2^{e+\tau-1}} 2^e\pfrac{2^{\tau-1}-1}{3}. \notag
\end{alignat} What remains to be shown is that 
$
\sum_{1\leq w < \tau} (-)^w2^{\less{w}}\brackets{T_w-{w\choose 0}} \equivMod{2^{e+\tau-1}} 0.
$
To this end, for each $k\in \naturals_0$, we will define 
$$\twoTailSum{2k+1} =  \sum_{1\leq w < 2k-1} (-)^w2^{w-1}\sum_{1\leq y < \ceilfrac{2k+1-w}{2}}{w+y\choose y}4^y;$$
we will show that
\[
 \sum_{1\leq w < \tau} (-)^w2^{\less{w}}\brackets{T_w-{w\choose 0}}  = 2^{e}\twoTailSum{\tau} =2^{e+\tau-1}\parentheses{\fibonacci{\tau-2} - 1}.
\] 

Assume the notation from the statement of Lemma \ref{lemma::FibonacciIdentity}. We will  demonstrate the chain of equalities 
\begin{equation}\label{equation::tailSumInductiveClaim}
\twoTailSum{2k+1} = \twoTailSum{2k-1}+ 4^{k-1}\twoSumTerm{k-1}  = 4^{k}\parentheses{\fibonacci{2k-1}-1}\notag
\end{equation} inductively for $k\in \naturals$.  Firstly, we have $\twoTailSum{3} = 0 + 4^0\twoSumTerm{0} = 4^{0}\parentheses{\fibonacci{1}-1} = 0$ for $k=1$.  
Assuming the inductive claim, we proceed with the chain of equalities for $k\geq 2$: 
\[
\twoTailSum{2k+1} =  \sum_{1\leq w < 2k-1} (-)^w2^{w-1}\sum_{1\leq y < \ceilfrac{2k+1-w}{2}}{w+y\choose y}4^y 
= \twoTailSum{2k-1}+ A_k,
\]
where
\[
A_k =\sum_{1\leq w < 2k-1} (-)^w2^{w-1}{w+\ceilfrac{2k-1-w}{2}\choose \ceilfrac{2k-1-w}{2}}4^{\ceilfrac{2k-1-w}{2}}.  
\]  The sum
\begin{alignat}{2}
A_k &= \sum_{1\leq w < 2k-1} (-)^w2^{w-1}{k+w+\ceilfrac{-1-w}{2}\choose k+\ceilfrac{-1-w}{2}}4^{k+\ceilfrac{-1-w}{2}}\notag\\
&= \sum_{1\leq w < \frac{2k-1}{2}} \brackets{2^{2w-1}{k+w\choose k-w} - 2^{2w-2}{k-1+w\choose k-w} }4^{k-w}\notag\\
&= 4^{k-1}\sum_{1\leq w < k} \brackets{2{k+w\choose k-w} - {k-1+w\choose k-w} }\notag\\
&= 4^{k-1}\sum_{1\leq w < k} \brackets{2{2k-w\choose w} - {2k-1-w\choose w} }\notag\\
&=  4^{k-1}\sum_{0\leq w <k-1} \brackets{2{2k-1-w\choose w+1} - {2k-2-w\choose w+1} }\notag\\
&= 4^{k-1}\twoSumTerm{k-1}.\notag
\end{alignat}

Thus, with Lemma \ref{lemma::FibonacciIdentity} and the inductive hypothesis, we can write
\[
\twoTailSum{2k+1} = \twoTailSum{2k-1} + 4^{k-1}\twoSumTerm{k-1} 
= 4^{k-1}\brackets{\fibonacci{2k-3} + \fibonacci{2k-2} + 3\fibonacci{2k-1}-4}
=4^{k}\brackets{\fibonacci{2k-1}-1}
\]as required.
Consequently, when $\tau \equivMod{2} 1 \equivMod{2}e-1$, the $\graded{2}$-remainder approximation 
$$\ladicResidueApproximation{\tau} = 2^e\pfrac{2^{\tau-1}-1}{3} + \frac{2^{e+\tau-1}-1}{3} + 2^{e+\tau-1}\parentheses{\fibonacci{\tau-2}-1} \equivMod{2^{e+\tau-1}}2^e\pfrac{2^{\tau-1}-1}{3} + \frac{2^{e+\tau-1}-1}{3}.$$ 
\hspace{20pt}
\end{proof}

Note that the approach within this subsection exploits the serendipitous pair of identities
$
3^{\eulerPhi{2}}-1 = 2
$ and 
$
2^{\eulerPhi{3}}-1 = 3.
$  
In general, Euler's Theorem allows one to write 
$
m^{\eulerPhi{l}}-1 = \resRepInv{-l}{m^{\eulerPhi{l}}}l,
$ and 
$
l^{\eulerPhi{m}}-1 =  \resRepInv{-m}{l^{\eulerPhi{m}}}m;
$ however, for arbitrary, coprime $m$ and $l$ exceeding $1$, the terms $\resRepInv{-l}{m^{\eulerPhi{l}}}$ and  $ \resRepInv{-m}{l^{\eulerPhi{m}}}$ may prevent one from executing the approach above in an analogous manner.

\subsection{Dual-Radix Modular Division}

The approach in this section, based on the work in \cite{Rukhin18},  demonstrates a different method of proving Theorem \ref{theorem::threeEcksPlusOneResidues} using {\sl dual-radix modular division}.

\begin{proof}
Under the assumption that
\begin{equation}
	\lExp{w} = 
		\begin{cases}
			1 & w \in \nZeroOpenSet{\tau-1}\\
			e & w = \tau-1,
		\end{cases}\notag
\end{equation} we have the following initial conditions for the recurrence in Theorem 4.4 in \cite{Rukhin18}.  For   $v\in \nZeroOpenSet{\tau}$, the  $3$-adic digit  $\madicDigit{v,0} \equivMod{3}\bracketsInverse{2^{\lExp{v}}}$; thus, we have
\begin{equation}
	\madicDigit{v,0} = 
		\begin{cases}
			2 & v \in \nZeroOpenSet{\tau-1}\\
			1 + e\bmod 2 & v = \tau-1;
		\end{cases}\notag
\end{equation} furthermore, the  $\graded{2}$-adic digit $\ladicDigit{v,0} \equivMod{2^{\lExp{v-1}}}\bracketsInverse{-3}$; thus, we have
\begin{equation}
	\ladicDigit{v,0}= 
		\begin{cases}
			\frac{2^{2\ceilfrac{e}{2}}-1}{3} & v = 0\\
			 1 & v \in \nSet{\tau-1}.
		\end{cases}\notag
\end{equation}
For $u>0$, the equivalences
\[
\madicDigit{v,u} \equivMod{3} \bracketsInverse{2^{\lExp{v}}}\brackets{\madicDigit{v+1,u-1} - \ladicDigit{v+u,u-1}}
\] and
\[
\ladicDigit{v,u} \equivMod{2^{\lExp{v-1-u}}} \bracketsInverse{-3}\brackets{\madicDigit{v-u,u-1} - \ladicDigit{v-1,u-1}}
\] yield, by induction on $u$, the equalities
$
\madicDigit{v,u} = 2[2-1] = 2
$ for $v< \tau-1-u$, and
$
\ladicDigit{v,u} = 1[2-1] = 1
$ for $v> u$.

Firstly, we will identify the $3$-adic digits of the $3$-remainder of $\iterate{0} =\iterate{\max}$.
When $e \equivMod{2} 1$, we have the initial condition $\madicDigit{\tau-1,0} = 2$. Thus, for $u\in \nOpenSet{\tau}$, the digit
$
\madicDigit{\tau-1-u, u} \equivMod{3} \bracketsInverse{2^{\lExp{\tau-1-u}}}\brackets{\madicDigit{\tau-u, u-1} - \ladicDigit{\tau-1, u-1}}
\equivMod{3} 2\brackets{2-1} 
\equivMod{3} 2, 
$ and thus we have $\madicDigit{0,\tau-1} = 2.$
Consequently, we have
$
\madicResidue{\tau} =  \sum_{0\leq w < \tau} 3^w\madicDigit{0, w }  = 3^{\tau}-1.
$

When $e \equivMod{2}0$, we have the initial condition $\madicDigit{\tau-1, 0} = 1$, and $\madicDigit{\tau-2,1} \equivMod{3} \bracketsInverse{2^1}\brackets{\madicDigit{\tau-1,0} - \ladicDigit{\tau-1,0}}\equivMod{3} \bracketsInverse{2^1}\brackets{1-1}\equivMod{3}0.$ 
By induction, for $u\in \nOpenSet{\tau}$ where  $u\equivMod{2}0$, the digit 
$$
\madicDigit{\tau-1-u,u} \equivMod{3} \bracketsInverse{2^{\lExp{ \tau-1-u}}}\brackets{\madicDigit{\tau-u, u-1} - \ladicDigit{\tau-1, u-1}}
\equivMod{3} 2\brackets{0-1} 
\equivMod{3} 1. 
$$
For  $u\equivMod{2}1$, the digit
$\madicDigit{\tau-1-u,u} \equivMod{3} \bracketsInverse{2^{\lExp{ \tau-1-u}}}\brackets{\madicDigit{\tau-u, u-1} - \ladicDigit{\tau-1, u-1}}
\equivMod{3} 2\brackets{1-1}
\equivMod{3} 0. 
$  Thus, the digit $\madicDigit{0,\tau-1} = \tau \bmod 2.
$  Thus, when $\tau \equivMod{2} 0$, the $3$-adic remainder
$
\madicResidue{\tau} = \sum_{0\leq w < \tau-1} 3^w(2) + 3^{\tau-1}(0) = 3^{\tau-1}-1;
$  and, when $\tau \equivMod{2} 1$, the $3$-adic residue
$
\madicResidue{\tau} =  \sum_{0\leq w < \tau-1} 3^w(2) + 3^{\tau-1}(1)  = 2\cdot3^{\tau-1}-1. 
$

We will now determine the $\graded{2}$-adic digits of $\iterate{} $ when $\tau\equivMod{2} 1\equivMod{2} e-1$:  the initial \graded{2}-adic digit 
$
\ladicDigit{0, 0} = \frac{2^e-1}{3},
$ and the digit
$
\ladicDigit{0,1} \equivMod{2^{\lExp{\tau-2}}} \bracketsInverse{-3}\brackets{\madicDigit{\tau-1, 0} - \ladicDigit{\tau-1, 0}}  \equivMod{2^{1}} (1)\cdot\brackets{1-1}  \equivMod{2^{1}} 0.
$  For $u\in \nOpenSet{\tau}$ where  $u\equivMod{2}0$, we have
$
\ladicDigit{0, u} \equivMod{2^{\lExp{\tau-1-u}}} \bracketsInverse{-3}\brackets{\madicDigit{\tau-u, u-1} - \ladicDigit{\tau-1, u-1}}  \equivMod{2^{1}} (1)\cdot\brackets{0-1}  \equivMod{2^{1}} 1, 
$ and, when $u\equivMod{2}1$, we have
$
\ladicDigit{0, u} \equivMod{2^{\lExp{\tau-1-u}}} \bracketsInverse{-3}\brackets{\madicDigit{\tau-u, u-1} - \ladicDigit{\tau-1, u-1}}  \equivMod{2^{1}} (1)\cdot\brackets{1-1}  \equivMod{2^{1}} 0. 
$ Thus, when $\tau \equivMod{2}1 \equivMod{2}e-1$, the \graded{2}-adic remainder
\begin{alignat}{2}
\ladicResidue{\tau}  &= \ladicDigit{0, 0} + \sum_{1\leq u < \tau}2^{\less{u}}\ladicDigit{0,u} \notag\\
&= \frac{2^e-1}{3} + 2^e\sum_{2\leq u < \tau}2^{u-1}[u \equivMod{2} 0] \notag \\
&= \frac{2^e-1}{3} + 2^{e+1}\sum_{0\leq u < \tau-2}2^{u}[u \equivMod{2} 0] \notag \\
&= \frac{2^e-1}{3} + 2^{e+1}\sum_{0\leq u \leq \frac{\tau-3}{2}}4^{u} \notag \\
&= \frac{2^e-1}{3} + 2^{e+1}\pfrac{4^{\frac{\tau-1}{2}}-1}{3} \notag\\
&= 2^e\pfrac{2^{\tau-1}-1}{3} + \frac{2^{e+\tau-1}-1}{3} .\notag
\end{alignat} 
\hspace{20pt}
\end{proof}

\subsection{Circuits in the $3x-1$ Dynamical System}

We conclude this article by applying the previous analyses to the $3x-1$ dynamical system; now, we will consider the case where $\translationValue{w} = -1$ for all $w\in \nZeroOpenSet{\tau}$.  

We will extend the argument in \cite{BelagaMignotte} to the case where $3^{\tau}>2^{\less{\tau}}:$  the magnitude of the numerator of a maximal iterate in a periodic orbit can be bound from above as follows:
\[
\left|\parentheses{2^{e}+1}3^{\tau-1} - 2^{\less{\tau}}\right| = 3^{\tau}\brackets{\frac{2^e+1}{3} - \frac{2^{\less{\tau}}}{3^{\tau}}} < 3^{\tau-1}\parentheses{2^{e}+1}.
\] We can bound the denominator $3^{\tau} - 2^{\less{\tau}}$ from below by appealing to the inequality (\ref{rhinInequality}) once again
to conclude that the maximal iterate \iterate{\max} within a periodic orbit  in the $3x-1$ dynamical system satisfies the inequality
\[
\iterate{\max} < \frac{\frac{2^e+1}{3}}{1-\frac{2^{e+\tau-1}}{3^{\tau}}} < \pfrac{2^e+1}{3}2\parentheses{e+\tau-1}^{13.3} = o(2^{e+\tau-1})
\] for any fixed $e\in \naturals$.  Thus, we will reuse the notation of  the previous section and begin with the following assumptions.
  
\begin{assumptions}[\arabic{chapter}.\arabic{section}.\arabic{assumptionCounter}]\label{hypotheses::threeEcksMinusOneResiduesOfMaxIterate} Assume \ref{hypotheses::threeEcksPlusOneResiduesOfMaxIterate}, except that now we assume that $N = 2^{e+\tau-1} - (2^e+1)3^{\tau-1}$, and  $D = 2^{e+\tau-1} - 3^{\tau} <0$.

As before, define $\madicResidue{\tau} = ND^{-1} \bmod 3^{\tau}$ and $\ladicResidue{\tau} = ND^{-1} \bmod 2^{e+ \tau -1}.$
\end{assumptions}


Our goal for the remainder of this subsection is to prove the following theorem:

\begin{theorem}\label{theorem::ThreeEcksMinusOneResidues}
Assume (\ref{hypotheses::threeEcksMinusOneResiduesOfMaxIterate}). 

The $3$-remainder
\begin{equation}
\madicResidue{\tau} = 
	\begin{cases}
		2\cdot 3^{\tau-1}+1 & e\equivMod{2} 0\\
		1& e\equivMod{2} 1
	\end{cases}\notag
\end{equation} when $\tau \equivMod{2} 0$, and
\begin{equation}
\madicResidue{\tau} = 
	\begin{cases}
		3^{\tau-1}+1 & e\equivMod{2} 0 \\
		1& e\equivMod{2} 1 
	\end{cases}\notag
\end{equation} when $\tau \equivMod{2} 1$.  

The $\graded{2}$-remainder
\begin{equation}
\ladicResidue{\tau} = 
	\begin{cases}
		\frac{2^e\parentheses{2^{\tau}+1}+1}{3} & e\equivMod{2} 0\\
		\frac{2^e+1}{3}& e\equivMod{2} 1
	\end{cases}\notag
\end{equation} when $\tau \equivMod{2} 0$, and
\begin{equation}
\ladicResidue{\tau} = 
	\begin{cases}
		\frac{2^e\parentheses{2^{\tau-1}+1}+1}{3} & e\equivMod{2} 0 \\
		\frac{2^e+1}{3}& e\equivMod{2} 1 
	\end{cases}\notag
\end{equation} when $\tau \equivMod{2} 1$. 
\end{theorem}


Analogous to Lemma \ref{Lemma::ThreeEcksPlusOneResidueCondition}, the following lemma will aid in identifying circuits within the $3x-1$ Dynamical System.

\begin{lemma}\label{Lemma::ThreeEcksMinusOneResidueCondition}
Assume that the $3$-remainder  is
\begin{equation}
\madicResidue{\tau} = 
	\begin{cases}
		2\cdot 3^{\tau-1}+1 & e\equivMod{2} 0\\
		1& e\equivMod{2} 1
	\end{cases}\notag
\end{equation} when $\tau \equivMod{2} 0$, and
\begin{equation}
\madicResidue{\tau} = 
	\begin{cases}
		3^{\tau-1}+1 & e\equivMod{2} 0 \\
		1& e\equivMod{2} 1 
	\end{cases}\notag
\end{equation} when $\tau \equivMod{2} 1$.  Moreover, assume that the $\graded{2}$-remainder is
\begin{equation}
\ladicResidue{\tau} = 
	\begin{cases}
		\frac{2^e\parentheses{2^{\tau}+1}+1}{3} & e\equivMod{2} 0\\
		\frac{2^e+1}{3}& e\equivMod{2} 1
	\end{cases}\notag
\end{equation} when $\tau \equivMod{2} 0$, and
\begin{equation}
\ladicResidue{\tau} = 
	\begin{cases}
		\frac{2^e\parentheses{2^{\tau-1}+1}+1}{3} & e\equivMod{2} 0 \\
		\frac{2^e+1}{3}& e\equivMod{2} 1 
	\end{cases}\notag
\end{equation} when $\tau \equivMod{2} 1$. 

The equality $\madicResidue{\tau} = \ladicResidue{\tau}$ holds if and only if either i.) $e = 1$ or  ii.) $e = \tau=2$.
\end{lemma}

\begin{proof}

When $e\equivMod{2} 1$, we require that the equality
$
\frac{2^e+1}{3} = 1
$ holds; consequently, we require that $e=1$ (irrespective of the parity of $\tau$).

When $e\equivMod{2} 0$ and $\tau \equivMod{2}0$, we require that the equality 
$
2\cdot 3^{\tau-1}+1 = \frac{2^e\parentheses{2^{\tau}+1}+1}{3}
$ holds. Equivalently, we require that
$
2\cdot 3^{\tau}+3= {2^e\parentheses{2^{\tau}+1}+1};
$ after simplifying, we require that
$
\frac{3^{\tau}+1}{2^{e-1}} = 2^{\tau}+1.
$  When $\tau\equivMod{2} 0$, the numerator on the left-hand side
$
9^{\frac{\tau}{2}}+1 \equivMod{4} 2;
$ thus, it follows that we require that $e=2$.  The equality
$
3^{\tau} = 2^{\tau+1}+1
$  holds only when $\tau=2$ as per a result of Gersonides\footnote{Levi Ben Gerson, 1342 AD. See \cite{gersonides}.} on {\sl harmonic numbers}.

When $e\equivMod{2} 0$ and $\tau \equivMod{2} 1$, we have $\madicResidue{\tau} \equivMod{2} 0$ and $\ladicResidue{\tau} \equivMod{2} 1$.
$\hspace{20pt}$
\end{proof}

\begin{proof}[Theorem \ref{theorem::ThreeEcksMinusOneResidues}]

We can write
\[
\madicResidue{\tau} \equivMod{3^{\tau}} N\bracketsInverse{2^{e+\tau-1}-3^{\tau}} 
 \equivMod{3^{\tau}} \brackets{2^{e+\tau-1}-(2^e+1)3^{\tau-1}}\resRepInv{2^{e+\tau-1}}{} 
  \equivMod{3^{\tau}}1 - \brackets{\resRepInv{2^{\tau-1}}{3^1} + \resRepInv{2^{e+\tau-1}}{3^1}} 3^{\tau-1}.
\]

As $\resRepInv{2^{u}}{3^1} \equivMod{3} \negPower{u}$ for $u\in \naturals$, it follows that
$
\madicResidue{\tau} \equivMod{3^{\tau}} 1 + 3^{\tau-1}\negPower{\tau}\brackets{1 + \negPower{e}}.
$
For the $\graded{2}$-remainder, we begin by writing
\[
\ladicResidue{\tau} \equivMod{2^{e+\tau-1}} N\bracketsInverse{2^{\less{\tau}}-3^{\tau}} 
 \equivMod{2^{e+\tau-1}}  \brackets{2^{e+\tau-1} - (2^e+1)3^{\tau-1} }\resRepInv{-3^{\tau}}{} 
  \equivMod{2^{e+\tau-1}} 2^e\resRepInv{3}{2^{\tau-1}} + \resRepInv{3}{2^{e+\tau-1}}.
\]

We will write $\resRepInv{3}{2^{\tau-1}} = \frac{2^{\tau-\parentheses{\tau-1}\bmod 2}+1}{3}$, and $\resRepInv{3}{2^{e+\tau-1}} = \frac{2^{e+\tau-\parentheses{e+\tau-1}\bmod 2}+1}{3}$, and we will  complete the proof by cases.
\begin{enumerate}[i.]
	\item ($e\equivMod{2} 0, \tau \equivMod{2} 0$) $\madicResidue{\tau} = 2\cdot3^{\tau-1}+1$, and $\ladicResidue{\tau} =  \brackets{2^e\pfrac{2^{\tau-1}+1}{3} + \frac{2^{e+\tau-1}+1}{3}} \bmod {2^{e+\tau-1}} = \frac{2^{e+\tau} + 2^e+1}{3}$
	\item ($e\equivMod{2} 0, \tau \equivMod{2} 1$) $\madicResidue{\tau} = 3^{\tau-1}+1$, and $\ladicResidue{\tau} =  \brackets{2^e\pfrac{2^{\tau}+1}{3} + \frac{2^{e+\tau}+1}{3}} \bmod {2^{e+\tau-1}} = \frac{2^{e+\tau-1} + 2^e+1}{3}$. 
	\item ($e\equivMod{2} 1, \tau \equivMod{2} 0$) $\madicResidue{\tau}  = 1$, and $\ladicResidue{\tau} =  \brackets{2^e\pfrac{2^{\tau-1}+1}{3} + \frac{2^{e+\tau}+1}{3}} \bmod {2^{e+\tau-1}} = \frac{2^e+1}{3}$. 
	\item  ($e\equivMod{2} 1, \tau \equivMod{2} 1$) $\madicResidue{\tau}  = 1$, and $\ladicResidue{\tau} =  \brackets{2^e\pfrac{2^{\tau}+1}{3} + \frac{2^{e+\tau-1}+1}{3}} \bmod {2^{e+\tau-1}} = \frac{2^e+1}{3}$.
\end{enumerate}

\end{proof}
	
Thus, under the assumption that $\iterate{} < 2^{e+\tau-1}$, the only circuits within the $3x-1$ dynamical system are $(1)$ and $(5,7)$.


\bibliographystyle{plain}

\end{document}